\documentclass[12pt]{amsart}

\usepackage{amssymb}
\usepackage{latexsym,array}
\usepackage{verbatim,euscript}
\usepackage{fullpage,color}

\input {cyracc.def}
\numberwithin{equation}{section}

\usepackage[colorlinks=true,linkcolor=blue,urlcolor=my_color,citecolor=magenta]{hyperref}

\definecolor{my_color}{rgb}{0,0.5,0.5}
\definecolor{MIXT}{rgb}{0.4,0.3,0.6}

 \font\tencyr=wncyr10 
\font\tencyi=wncyi10 
\font\tencyb=wncyb10 
\font\tencysc=wncysc10 
\def\rus{\tencyr\cyracc}
\def\rusi{\tencyi\cyracc}
\def\rusc{\tencysc\cyracc}
\def\rusb{\tencyb\cyracc}

\newtheorem{thm}{Theorem}[section]

\newtheorem{lm}[thm]{Lemma}

\newtheorem{prop}[thm]{Proposition}

\theoremstyle{remark}
\newtheorem{rmk}[thm]{Remark}

\theoremstyle{definition}
\newtheorem{ex}[thm]{Example}
\newtheorem*{ex-bn}{Example}
\newtheorem*{rema}{Remark}

\newtheorem{quest}{Question}

\newenvironment{proof*}
{\noindent {\sl Proof.}\quad }{\hfill
$\square$}

\newcommand {\ah}{{\mathfrak a}}

\newcommand {\ee}{{\mathfrak e}}

\newcommand {\g}{{\mathfrak g}}
\newcommand {\h}{{\mathfrak h}}

\newcommand {\el}{{\mathfrak l}}

\newcommand {\n}{{\mathfrak n}}
\newcommand {\fN}{{\mathfrak N}}
\newcommand {\p}{{\mathfrak p}}

\newcommand {\es}{{\mathfrak s}}
\newcommand {\te}{{\mathfrak t}}
\newcommand {\ut}{{\mathfrak u}}
\newcommand {\z}{{\mathfrak z}}

\newcommand {\sln}{\mathfrak{sl}_n}

\newcommand {\glv}{\mathfrak{gl}(\VV)}
\newcommand {\slno}{\mathfrak{sl}_{n+1}}

\newcommand {\eus}{\EuScript}

\newcommand {\ap}{\alpha}

\newcommand {\lb}{\lambda}
\newcommand {\vp}{\varphi}

\newcommand {\ca}{{\mathcal A}}

\newcommand {\sfr}{{\eus R}}

\newcommand {\VV}{{\mathbb V}}
\newcommand {\UU}{{\mathbb U}}

\newcommand {\BZ}{{\mathbb Z}}
\newcommand {\BN}{{\mathbb N}}
\newcommand {\BQ}{{\mathbb Q}}
\newcommand {\BR}{{\mathbb R}}

\newcommand {\zge}{{\g^e}}
\newcommand {\zgen}{{\g^e_{nil}}}

\newcommand {\ad}{{\mathrm{ad\,}}}

\newcommand {\End}{{\mathsf{End\,}}}

\newcommand {\hot}{{\mathsf{ht}}}

\newcommand {\Lie}{{\mathsf{Lie}}}
\newcommand {\Ker}{{\mathrm{Ker\,}}}

\newcommand {\rk}{{\mathsf{rk}}}

\newcommand {\tri}{\mathfrak{sl}_2}
\newcommand {\GR}[2]{{\textrm{{\bf #1}}}_{#2}}

\newcommand {\un}{\underline}

\newcommand {\beq}{\begin{equation}}
\newcommand {\eeq}{\end{equation}}

\renewcommand{\le}{\leqslant}
\renewcommand{\ge}{\geqslant}

\newcommand{\ri}{wide}
\newcommand{\fxpl}{\mathfrak X_+}

\newcommand{\vlb}{\sfr(\lb)}

\newcommand {\bbk}
{\mathbb C}%

\begin{document}
\setlength{\parskip}{1pt plus 2pt minus 0pt}
\hfill {\scriptsize May 5, 2013}
\vskip1.5ex

\title[Wide subalgebras]
{Wide subalgebras of semisimple Lie algebras}
\author[D.\,Panyushev]{Dmitri I. Panyushev}
\address
{Institute for Information Transmission Problems of the R.A.S., 
\hfil\break\indent Dobrushin Mathematical Laboratory,
\hfil\break\indent B. Karetnyi per. 19, Moscow 
127994, Russia}
\email{panyushev@iitp.ru}

\maketitle

\section*{Introduction}
\noindent
Let  $G$ be a connected semisimple algebraic group over $\bbk$, with Lie algebra $\g$. Let $\h$ be a
subalgebra of $\g$.
A simple finite-dimensional $\g$-module $\VV$ is said to be $\h$-{\it indecomposable\/} if it
cannot be written as a direct sum of two proper $\h$-submodules. We say that $\h$ is {\it \ri\/}, 
if all simple finite-dimensional $\g$-modules are $\h$-indecomposable.
Some very special examples of indecomposable modules and \ri\ subalgebras appeared
recently in the literature, see ~\cite{casati,duglas} and references therein.
In this paper, we  point out several large classes of \ri\ subalgebras of $\g$ and 
initiate their systematic study. 

Our approach relies on the following simple  observation. 
Suppose that $\VV=V_1\oplus V_2$ is a
sum of two nontrivial $\h$-modules. Let $p:\VV\to V_1\subset \VV$ be the projection along $V_2$.
Then $p$ is a nontrivial idempotent in the associative algebra, $(\mathsf{End\,}\VV)^\h$, of $\h$-invariant elements in $\End\VV$.
Consequently, 
\[
\left\{\text{\parbox{4.3cm}{$\VV$ is  $\h$-indecomposable}}\right\} 
\quad \Longleftrightarrow \quad 
\left\{ \text{\parbox{4.7cm}{$(\mathsf{End\,}\VV)^\h$ does not contain non-trivial idempotents}}\right\} . 
\]

\noindent
The map $\mathsf{Id}_\VV:\VV\to \VV$ is the unit in the associative algebra $(\mathsf{End\,}\VV)^\h$, and we repeatedly use the following  sufficient condition for the absence of non-trivial idempotents in $(\mathsf{End\,}\VV)^\h$:
\vskip.7ex
\noindent
\hfil \parbox{.89\textwidth}{%
\sl Suppose that $(\mathsf{End\,}\VV)^\h=\bigoplus_{i\in\BN}(\mathsf{End\,}\VV)^\h(i)$ 
is graded (as associative algebra!) and
$(\mathsf{End\,}\VV)^\h(0)=\bbk{\cdot}\mathsf{Id}_\VV$. Then $(\mathsf{End\,}\VV)^\h$ does 
not contain non-trivial idempotents.
}\hfil
\vskip.8ex\noindent
We prove that such a grading exists for every simple $\g$-module $\VV$ if $\h$ belongs to the following list:

({\bf A}) \  
{\it  $\p\subset\g$ is a parabolic subalgebra that contains no simple ideals of\/ $\g$, and 
$\h$ is the nilradical of\/ $\p$; in particular, if\/ $\g$ is simple, then $\p$ can be any proper parabolic subalgebra}
(Section~\ref{sect:radical});

{\it ({\bf B}) \  $e\in\g$ is a nilpotent element that has a non-trivial projection to any simple ideal of\/ $\g$ and 
$\h$ is the nilradical of the centraliser of $e$; in particular, if\/ $\g$ is simple, then $e$ can be any 
nonzero nilpotent element} (Section~\ref{sect:reg-nilp});

({\bf C}) \ 
{\it $\h$ is a certain subalgebra that consists of nilpotent elements of\/ $\g$ (= \textsf{ad}-nilpotent 
subalgebra) and is normalised by a Cartan
subalgebra of\/ $\g$.} 
For a sensible description, we use the standard notation 
on root systems, see also \ref{notation} below.
Let $\te$ be a Cartan subalgebra of $\g$, $\Delta$ the root system of $(\g,\te)$, and $\g_\gamma$ 
the root space of $\g$ corresponding to $\gamma\in\Delta$. If $[\te,\h]\subset \h$ and $\h$ is 
\textsf{ad}-nilpotent, then $\h=\bigoplus_{\gamma\in\Delta_\h}\g_\gamma$, where $\Delta_\h$ is a 
closed subset of $\Delta$
and $\Delta_\h\cap (-\Delta_\h)=\varnothing$. The main result of Section~\ref{sect:pi-part} asserts that
$\h$ is \ri\ {\sl if and only if\/}  the closure of $\Delta_\h\cup (-\Delta_\h)$ is the whole root system $\Delta$.
 
(C$_1$)\ \ A special case of this construction is a subalgebra  
determined by a partition of a set of simple roots $\Pi$ in $\Delta$.
Let $\Pi'$ be a subset of $\Pi$.
Define $\h=\h(\Pi')$ to be the subalgebra of $\g$ generated by $\g_\ap$ ($\ap\in\Pi'$) and
$\g_{-\ap}$ ($\ap\in\Pi\setminus\Pi'$). We say that $\h(\Pi')$ is a  $\Pi$-{\it partition subalgebra} of $\g$.
Clearly, $\dim\h(\Pi')\ge \#\Pi$.
It is easily seen that $\h(\Pi')\simeq \h(\Pi\setminus\Pi')$ and
 $\h(\Pi')$ is \textsf{ad}-nilpotent. 
There is a special subset $\tilde\Pi\subset\Pi$ such that $\h(\tilde\Pi)$ is abelian
and $\dim\h(\tilde\Pi)=\#\Pi$. Namely, $\tilde\Pi$ is a set of pairwise orthogonal simple roots such
that $\Pi\setminus\tilde\Pi$ also consists of pairwise orthogonal roots. 
Since the Dynkin diagram is a tree, the partition $\Pi=\tilde\Pi\sqcup (\Pi\setminus \tilde\Pi)$ is 
unique, and in this case the vector space 
\[
    \bigl(\bigoplus_{\ap\in\tilde\Pi}\g_\ap\bigr)\oplus
    \bigl(\bigoplus_{\ap\in\Pi\setminus\tilde\Pi}\g_{-\ap}\bigr) \ \text{ or } \
    \bigl(\bigoplus_{\ap\in\tilde\Pi}\g_{-\ap}\bigr)\oplus
    \bigl(\bigoplus_{\ap\in\Pi\setminus\tilde\Pi}\g_{\ap}\bigr)
\]
is already an (abelian) subalgebra of dimension $\#\Pi$. It was proved in \cite{casati} that 
$\h(\tilde\Pi)$ is wide for $\g=\slno$.  Our proof is much easier and yields a more general
assertion.

(C$_2$)\ \ Another possibility is to take $\tilde\ut=[\ut^+,\ut^+]$, where $\ut^+$ is a maximal 
nilpotent subalgebra of $\g$. Here $\Delta_{\tilde\ut}=\Delta^+\setminus\Pi$, and 
the closure of $\Delta_{\tilde\ut}\cup (-\Delta_{\tilde\ut})$ equals $\Delta$ if and only if $\g$ 
has no simple ideals $\tri$ or $\mathfrak{sl}_3$.
Invariant-theoretic properties of $\tilde\ut$ have been studied in \cite{imrn10}.

In Section~\ref{sect:general}, we gather simple general properties of wide subalgebras and discuss
a relationship between wide subalgebras and epimorphic subgroups. 
A subgroup of $H\subset G$ is  {\it epimorphic\/} if the following condition holds: 
{\sl If\/ $\VV$ is a finite-dimen\-si\-onal rational $G$-module and\/ $\VV=V_1\oplus V_2$ is a direct sum 
of\/ $H$-modules, then the subspaces 
$V_1,V_2$ are actually $G$-invariant} (see \cite{bb92}). For a simple
$G$-module $\VV$, this is just the $H$-indecomposability condition. Therefore, if $H$ is epimorphic, then $\Lie(H)$ is wide. However, our work shows that there are much more wide subalgebras than Lie algebras of epimorphic subgroups. Indeed, epimorphic subgroups are also characterised by the property that
$\bbk[G]^H=\bbk$, hence they cannot be unipotent, whereas all
wide subalgebras described in (A),\,(B), and (C) are \textsf{ad}-nilpotent. 
We also give an example of a two-dimensional wide subalgebra of $\g$ and provide a quick 
derivation (and generalisation) for the results of \cite{duglas}.
\vskip1ex
{\small {\bf Acknowledgements.}
I am grateful to Alexander Premet (Manchester) for drawing my attention to the role of idempotent 
elements in the associative algebra $(\End\VV)^\h$.}

\section{Notation and other preliminaries}
\label{sect:notat-prelim}

\noindent 
\subsection{Notation}   \label{notation}
We fix a triangular decomposition
$\g=\ut^+\oplus\te\oplus\ut^-$ and various objects associated with the root system
$\Delta=\Delta(\g,\te)$. Specifically,

{\bf --} \quad $\Delta^+$ is the set of positive roots (= the roots of $\ut^+$);

{\bf --} \quad $\Pi=\{\ap_1,\dots,\ap_n\}$ is the set of simple roots in $\Delta^+$;

{\bf --} \quad $\{\vp_\ap \mid \ap\in\Pi\}$ are the fundamental weights and
$\fxpl$ is the set of dominant weights
corresponding to $\Pi$; 

{\bf --} \quad $\eus Q=\bigoplus_{i=1}^n \BZ\ap_i$ is the root lattice, $\eus E=\eus Q\otimes_\BZ\BR$, 
and $\eus P$ is the weight lattice in $\eus E$.

{\bf --} \quad $(\ ,\ )$ is a Weyl group invariant inner product in $\te$. Using this inner product, we identify
$\te$ and $\te^*$, and regard $\eus E$ as a real form of $\te$. \\
For any $\gamma\in\Delta$, let $\g_\gamma$ denote the corresponding root space. 
We also fix a nonzero element $e_\gamma\in\g_\gamma$. All $\g$-modules are assumed to be 
finite-dimensional. Write $\z_\g(M)$ or $\g^M$ for the centraliser of a subset $M\subset \g$.

\subsection{Rational semisimple elements and gradings}  \label{subs:h-grad}
Let $h\in\g$ be a rational semisimple element, i.e., the eigenvalues of $h$ in $\g$ are rational.
Then $h$ has rational eigenvalues in any finite-dimensional $\g$-module $\VV$. Therefore, 
\beq   \label{eq:h-grad}
   \VV=\bigoplus_{i\in\BQ} \VV_h(i) ,
\eeq
where $\VV_h(i)=\{ v\in \VV\mid \rho_{\VV}(h){\cdot}v=iv\}$ and $\rho_{\VV}:\g\to \End\VV=\glv$ is the representation.
We also say that \eqref{eq:h-grad} is the $h$-{\it grading\/} of $\VV$. Each subspace $\VV_h(i)$
is $\g^h$-stable.

\begin{lm}  \label{lm:ass-alg-grad} 
Let $h\in\g$ be a rational semisimple element.
Given a $\g$-module $\VV$, consider the $h$-grading of the $\g$-module $\End\VV$, \ 
$     \End\VV=\bigoplus_{i\in\BQ} (\End\VV)_h(i)$. Then
\begin{itemize}
\item[\sf (i)]  \ this is an associative algebra grading;
\item[\sf (ii)] \ if\/ $\h$ is a subalgebra of $\g$, then $(\End\VV)^\h$ is an associative subalgebra
of\/ $\End\VV$. Moreover, if $[h,\h]\subset \h$, then $(\End\VV)^\h$ inherits the $h$-grading.
\end{itemize}
\end{lm}
\begin{proof}
(i) The $\g$-module structure in $\End\VV$ is given by 
\[
    (x,A) \mapsto \rho_{\VV}(x)A-A\rho_{\VV}(x)=[\rho_{\VV}(x),A], \quad x\in\g, A\in\End\VV .
\]
If $[\rho_{\VV}(h),A]=iA$ and $[\rho_{\VV}(h),B]=jB$ with $i,j\in\BQ$, then
\[
  [\rho_{\VV}(h),AB]=\rho_{\VV}(h)AB-AB\rho_{\VV}(h)=\bigl(iA+A\rho_{\VV}(h)\bigr)B-A\bigl(\rho_{\VV}(h)B-jB\bigr)=(i+j)AB .
\]
(ii) \ Similarly.
\end{proof}

\begin{lm}  \label{lm:no-idempot}
Let $\ca$ be a finite-dimensional $\BN$-graded unital associative algebra,  
$\ca=\bigoplus_{i\in\BN}\ca(i)$. Suppose that $\ca(0)=\bbk{\cdot}I$, where $I$ is the unit.
Then $I$ is the only idempotent of $\ca$. 
\end{lm}
\begin{proof}
Any $p\in\ca$ can be written as $p=cI+q$, where $c\in\bbk$ and $q\in \bigoplus_{i\ge 1}\ca(i)$.
If $p^2=p$, then $c=1$ and $q^2+q=0$. As  $q^n=0$ for $n\gg 0$, $1+q$ is invertible and
$q=0$.
\end{proof}

{\bf Warning.}  If $\dim\ca(0)\ge 2$, then $\ca$ may have non-trivial idempotents that are not contained in $\ca(0)$.
\\[.7ex]
We also need a slightly different version that concerns the case in which $(\End\VV)^\h$ is 
positively multigraded. If $[\te,\h] \subset \h$, then the associative algebra $(\End\VV)^\h$
is being decomposed in a finite sum of $\te$-weight spaces,
\beq   \label{eq:multigrad}
       (\End\VV)^\h=\bigoplus_{\nu\in\eus P}(\End\VV)^\h_\nu .
\eeq
\begin{lm} \label{lm:no-idemp-multi}
Suppose that the set $\eus P(\VV,\h)=\{\nu\in\eus P\mid (\End\VV)^\h_\nu\ne 0\}$ is 
contained in a closed strictly convex cone $\eus C\subset\eus E$ and 
$(\End\VV)^\h_0=\bbk{\cdot}\mathsf{Id}_\VV$. Then $\mathsf{Id}_\VV$ is the only idempotent in
$(\End\VV)^\h$ and $\VV$ is $\h$-indecomposable.
\end{lm}
\begin{proof}
Let $h\in\te$ be a rational element such that $\mu(h)>0$ for all $\mu\in\eus C\setminus\{0\}$
and $\mu(h)\in \BZ$ for all $\mu\in\eus P(\VV,\h)$. Then \eqref{eq:multigrad} can be specialised 
to the $h$-grading, where Lemmas~\ref{lm:ass-alg-grad} and \ref{lm:no-idempot} apply.
Alternatively, one can directly prove that \eqref{eq:multigrad} is an associative algebra grading
and the argument of Lemma~\ref{lm:no-idempot} goes through for positive multigradings.
\end{proof}

\begin{rmk}
For future use, we recall the standard fact that if $\VV$ is a simple $\g$-module, then
all $\te$-weights of the $\g$-module $\End\VV$ belong to the root lattice $\eus Q$.
\end{rmk}

\section{The nilradical of a proper parabolic subalgebra is wide}
\label{sect:radical}

\noindent
Let $\Pi'$ be an arbitrary subset of $\Pi$. 
If $\gamma=\sum_{\ap\in\Pi} a_\ap\ap\in\Delta$, then  $\hot_{\Pi'}(\gamma)=\sum_{\ap\in\Pi'}a_\ap$ is called the $\Pi'$-{\it height\/} of $\gamma$. For $\Pi'=\Pi$, one obtains the usual notion of the height.

Let $\p$ be the standard parabolic subalgebra of $\g$ determined by  $\Pi'\subset\Pi$.
That is, 
\[
   \p=\te\oplus \bigl(\bigoplus_{\gamma:\ \hot_{\Pi'}(\gamma)\ge 0}\g_\gamma \bigr) .
\]
Then \qquad 
$\displaystyle \p_{nil}=\n=\bigoplus_{\gamma:\ \hot_{\Pi'}(\gamma)> 0}\g_\gamma\quad \text{ is the nilpotent radical of $\p$, and}$
\[ 
 \el=\te\oplus \bigl(\bigoplus_{\gamma:\ \hot_{\Pi'}(\gamma)= 0}\g_\gamma \bigr) \quad \text{ is
 the standard Levi subalgebra of $\p$}.
\]

\begin{lm}   \label{lm:n-invar}
 If\/ $\VV$ is a simple $\g$-module, then $\VV^\n$ is a simple $\el$-module.
\end{lm}
\begin{proof}
If $\ut(\el)$ is an arbitrary maximal nilpotent subalgebra of $\el$, then $\ut(\el)\oplus\n$ is a
maximal nilpotent subalgebra of $\g$. Therefore,
$ \dim(\VV^\n)^{\ut(\el)}=\dim \VV^{\ut(\el)\oplus\n}=1$.
\end{proof}

We extend the $\Pi'$-height to the whole of $\eus P$, using the same formulae as above. That is,
if $\nu=\sum_{\ap\in\Pi} b_\ap\ap\in\eus P$, then  $\hot_{\Pi'}(\nu)=\sum_{\ap\in\Pi'}b_\ap$.
The coefficients $b_\ap$ and hence $\hot_{\Pi'}(\nu)$ can be rational. More precisely,  
$b_\ap\in \frac{1}{f}\BZ$, where $f=[\eus P:\eus Q]$ is the index of connection of $\Delta$.
In this way, one obtains the canonical {\it grading of type\/} $\Pi'$ in any $\g$-module $\VV$.
Namely,
\beq   \label{eq:Z-grad-W}
     \VV=\bigoplus_{i\in\frac{1}{f}\BZ} \VV(i),
\eeq
where $\VV(i)$ is the sum of weight spaces of $W$ corresponding to the weights of 
$\Pi'$-height $i$. 
Obviously, 
\beq    \label{eq:sovpad}
   \hot_{\Pi'}(\nu)=(\sum_{\ap\in\Pi'}\vp_\ap^\vee,\nu) ,
\eeq
where $\vp_\ap^\vee=2\vp_\ap/(\ap,\ap)$. Therefore, the grading of type $\Pi'$ is nothing but the
$h$-grading in the sense of Subsection~\ref{subs:h-grad}, with $h=\sum_{\ap\in\Pi'}\vp_\ap^\vee\in \te^*\simeq \te$.

Clearly, each $\VV(i)$ is an $\el$-module and 
$\g_\ap{\cdot}\VV(i)\subset \VV(i{+}1)$
if $\ap\in\Pi'$, i.e., $\hot_{\Pi'}(\ap)=1$. 
If $\VV$ is a simple $\g$-module, then all $i\in\BQ$ such that $\VV(i)\ne 0$ give rise to one and
the same element in $\BQ/\BZ$. 
Moreover, if all the weights of $\VV$ belong to $\eus Q$, then the grading of type $\Pi'$ is
a $\BZ$-grading on $\VV$. 

To avoid a cumbersome notation, we assume below that $\g$ is simple
(see also Remark~\ref{rem:simple-semisimple}). Let $\p$ be a proper parabolic subalgebra, i.e., $\Pi'\ne
\varnothing$.

\begin{lm}    \label{lm:high-irr}
Let $\VV$ be a simple $\g$-module equipped with the canonical grading of type $\Pi'$
\eqref{eq:Z-grad-W}. Set $m=\max\{ i \mid \VV(i)\ne 0\}$.
Then {\sf (i)} $\VV^\n=\VV(m)$ and $m\ge 0$; {\sf (ii)} $m=0$ if and only if\/ $\VV$ is a trivial 
one-dimensional module.
\end{lm}
\begin{proof}
Let $\lb\in \fxpl$ be the highest weight of $\VV$, and $\lb=\sum_{a\in\Pi} c_\ap\ap$.
Clearly, $m=\hot_{\Pi'}(\lb)=\sum_{\ap\in\Pi'} c_\ap$. Here \un{all} the coefficients $c_\ap$
are strictly positive if $\lb\ne 0$. (This follows from the fact that \un{all} the entries of the inverse of the Cartan 
matrix of $\Delta$ are strictly positive~\cite{lt92}.) Hence if $\lb\ne 0$, then $\hot_{\Pi'}(\lb) >0$ for any 
non-empty $\Pi'$.

Since $\g_\gamma{\cdot}\VV(i)\subset \VV(i{+}\hot_{\Pi'}(\gamma))$
for any $\gamma\in\Delta^+$, we have $\VV(m)\subset \VV^\n$. On the other hand, $\VV^\n$
is a simple $\el$-module (Lemma~\ref{lm:n-invar}), hence $\VV(m)= \VV^\n$.
\end{proof}

\begin{thm}   \label{thm:main1-pos-grad}
For any nonempty subset $\Pi'\subset\Pi$ and any
simple finite-dimensional $\g$-module $\VV$, 
\begin{itemize}
\item[\sf (i)] \  the grading of type $\Pi'$ on $(\End\VV)^\n$ is actually an $\BN$-grading 
and $(\End\VV)^\n(0)=\bbk{\cdot}\mathsf{Id}_\VV$. 
\item[\sf (ii)] \  $(\End\VV)^\n$ contains no non-trivial idempotents and, therefore,
$\n$ is a \ri\ subalgebra of $\g$.
\end{itemize}
\end{thm}
\begin{proof}
(i) Since the weights of the $\g$-module $\End\VV$ belong to the root lattice $\eus Q$,
the grading of type $\Pi'$ on $\End\VV$ is actually a $\BZ$-grading. 
We have $\End\VV\simeq \VV\otimes\VV^*=\sum_{j=1}^k \VV_j$, where $\VV_j$ are certain
simple $\g$-modules. 
If $\VV_j=\bigoplus_{i\in\BZ} \VV_j(i)$ be the $\BZ$-grading of type $\Pi'$ and
$m_j:=\max\{ i \mid \VV_j(i)\ne 0\}$, then
\beq     \label{eq:n-grading-EndVn}
   (\End\VV)^\n=\bigoplus_{j=1}^k \VV_j(m_j) 
\eeq
is the direct sum of simple $\el$-modules and also a (refinement of) $\BN$-grading. By the Schur 
lemma, $\VV\otimes \VV^*$ contains a unique trivial one-dimensional $\g$-module, and this unique
trivial module is the line through $\mathsf{Id}_\VV:\VV\to \VV$.
In view of Lemma~\ref{lm:high-irr}, we may assume that
$m_1=0$ and $m_j>0$ for $j\ge 2$. 

(ii)  The grading of type $\Pi'$ in $\End\VV$ is also the $(\sum_{\ap\in\Pi'}\vp_\ap^\vee)$-grading
(see Eq.~\eqref{eq:sovpad}). Then
Lemma~\ref{lm:ass-alg-grad}(i) guarantees us that this is an associative algebra grading.
Furthermore, $\te$ normalises $\n$ and  $\sum_{\ap\in\Pi'}\vp_\ap^\vee$ is identified with an
element of $\te$. Therefore, \eqref{eq:n-grading-EndVn} is also an associative algebra grading and,
by Lemma~\ref{lm:no-idempot}, $\mathsf{Id}_\VV$ is the only idempotent in 
$(\End\VV)^\n$.
\end{proof}

\begin{rmk}  \label{rmk:min-nil-rad}
It is known that $\dim\n =\dim G/P\ge \rk(\g)$, and the equality only occurs for
the maximal parabolic subalgebra of
$\slno$ such that $\Pi'=\{\ap\}$ and $\ap$ is an extreme root in the Dynkin diagram,
see \cite{vi72}. In particular, if $\dim\n=\rk(\g)$, then $\n$ is abelian.
\end{rmk}

\begin{rmk}   \label{rem:simple-semisimple}
If $\g$ is semisimple but not simple, then $\g=\prod_{j}\g^{(j)}$ is the product of simple ideals and 
$\Pi=\bigcup_j \Pi^{(j)}$. It is then easily seen that Lemma~\ref{lm:high-irr} and Theorem~\ref{thm:main1-pos-grad}
remain true if $\Pi'\cap \Pi^{(j)}\ne \varnothing$ for all $j$, i.e., if $\p$ does not contain simple ideals of $\g$.
\end{rmk}
 
\section{The nilradical of the centraliser of a non-degenerate nilpotent element is \ri}
\label{sect:reg-nilp}

\noindent
Let $\fN$ be the set of all nilpotent elements of $\g$.  Throughout this section, we assume that 
$e\in\fN$ is nonzero. To present a (well-known) description of the nilpotent radical of $\zge$, we need the machinery of $\tri$-triples and respective $\BZ$-gradings of $\g$.
By the Morozov-Jacobson theorem, any nonzero $e\in\fN$ can  be embedded
into an $\tri$-triple $\{e,h,f\}$ (i.e., $[h,e]=2e$, $[h,f]=-2f$, $[e,f]=h$) \cite[3.3]{CM}. The eigenvalues 
of  $h$ in any $\g$-module are integral, hence the $h$-grading in any simple $\g$-module is actually
a $\BZ$-grading.

As in Subsection~\ref{subs:h-grad}, the semisimple element $h$ 
determines the $h$-grading of $\g$:
\[
 \g=\bigoplus_{i\in\BZ}\g_h( i)   ,
\]
where $\g_h( i) =\{\,x\in\g\mid [h,x]=ix\,\}$. Then $e\in\g_h(2)$ and $f\in\g_h(-2)$.

The following facts on the structure of this grading and the centraliser
$\g^e$ are standard, see \cite[ch.\,III, \S\,4]{ss} or \cite[Ch.\,3]{CM}.
\begin{prop}      \label{prop:stab}
Let $\{e,h,f\}$ be an $\tri$-triple. Then
\begin{itemize}
\item[\sf (i)] \ the Lie algebra $\g^e$ is non-negatively
graded: $\g^e=\bigoplus_{i\ge 0} \g^e_h( i) $, where $\g^e_h( i) =\g^e\cap\g_h( i) $. 
Here $\g^e_{nil}:=\bigoplus_{i\ge 1}\g^e_h(i)$ is the nilpotent  radical 
and $\g^e_{red}:=\g^e_h(0)$ is a Levi subalgebra of $\g^e$; actually, $\g^e_h(0)=\g^{\{e,h,f\}}$.
\item[\sf (ii)] \ 
$\ad e:\g_h( i-2)\to\g_h( i) $ is injective
for $i\le 1$ and surjective for $i\ge 1$;
\item[\sf (iii)] \ $\dim\g^e=\dim \g_h( 0)+\dim \g_h( 1) $ and\/
$\dim\g^e_{nil}=\dim \g_h( 1)+\dim \g_h( 2) $.
\end{itemize}
\end{prop}

This provides a rather good understanding of the nilpotent radical $\zgen$.
Recall that $e$ is said to be {\it principal}, if $\dim\g^e=\rk(\g)$, and then $\g^e=\zgen$.
In this case we also say that $\langle e,h,f\rangle$ is a principal $\tri$-subalgebra.

If $\g=\g^{(1)}\oplus\g^{(2)}$ is a sum of two ideals and $e=e_1+e_2$, with $e_i\in\g^{(i)}$, then
$\g^e=(\g^{(1)})^{e_1}\oplus (\g^{(2)})^{e_2}$. Therefore, $(\g^{(i)})^{e_i}_{nil}=0$ if and only if
$e_i=0$. We say that $e\in\fN$ is {\it non-degenerate}, if $e$ has a non-trivial projection to every
simple ideal of $\g$. If $\g$ is simple, then any nonzero $e\in\fN$ is non-degenerate.

\begin{lm}  \label{lm:generated}
Suppose that $e\in\fN$ is non-degenerate. Then
the subalgebra generated by $\zgen$ and $f$ is the whole of  $\g$.  
\end{lm}
\begin{proof}
Set $\es=\langle e,h,f\rangle$. It is a three-dimensional simple subalgebra of $\g$.
Consider $\g$ as $\es$-module.
By Proposition~\ref{prop:stab}(i), $\zgen$ is the linear span of the highest vectors of  all nontrivial
simple $\es$-modules in $\g$.
Therefore, the minimal $(\ad f)$-stable subspace containing $\zgen$,
say $\UU$, is the sum of all nontrivial $\es$-submodules in $\g$, and the subalgebra generated by 
$\zgen$ and $f$ coincides with the subalgebra generated by $\UU$. 
The reductive algebra $\g^e_h(0)=\g^\es$ is the sum of all trivial $\es$-modules. Hence
$\UU\oplus \g^\es=\g$ and $[\g^\es,\UU]\subset\UU$. Let $\langle\UU\rangle$ be the subalgebra
generated by $\UU$. Then $[\g^\es,\langle\UU\rangle]\subset\langle\UU\rangle$ and
$[\UU,\langle\UU\rangle]\subset\langle\UU\rangle$. Hence $\langle\UU\rangle$ is an ideal of $\g$.
By the asumption, $\langle\UU\rangle$ has non-trivial projections to all simple ideals of $\g$. Hence
$\langle\UU\rangle=\g$ (cf. also \cite[Lemma\,4.1]{kac80}).
\end{proof}

A simple $\g$-module with highest weight $\lb\in\fxpl$ is denoted by $\vlb$, and 
$\rho_\lb$ is the corresponding representation of $\g$. 
If $\ah$ is any subset of $\g$, then 
\[
  \vlb^\ah=\{v\in\vlb \mid \rho_\lb(x)v=0 \ \ \forall x\in\ah\} .
\]
In particular, $\vlb^h$ is the zero weight space of the $h$-grading of $\vlb$.

\begin{prop}     \label{prop:even-normal}
Let $\{e,h,f\}$ be an $\tri$-triple in $\g$. If $e$ is non-degenerate and $\lb\ne 0$, then 
the $h$-eigenvalues in  $\vlb^\zgen$ are strictly positive.
\end{prop}
\begin{proof}  
It follows from the theory of $\tri$-representations that the $h$-eigenvalues in $\vlb^e$ are 
nonnegative. Hence the same is true for $\vlb^\zgen\subset \vlb^e$, and our goal is to prove that 
$0$ does not occur
as an $h$-eigenvalue in $\vlb^\zgen$. If
$v\in\vlb^h\cap \vlb^\zgen$, then $v$ is killed by both $h$ and $e$. Therefore, 
$\rho_\lb(f)(v)=0$. Thus, $v$ is killed by $f$ and $\zgen$. By Lemma~\ref{lm:generated}, the 
subalgebra generated by $f$ and $\zgen$ is $\g$. Hence $v\in\vlb^\g=\{0\}$.
\end{proof}

Now, we are ready to prove the main result of this section.

\begin{thm}   \label{thm:zge-pos-grad}
For any non-degenerate $e\in\fN$ and any simple finite-dimensional $\g$-module $\vlb$, we have
\begin{itemize}
\item[\sf (i)] \  
the $h$-grading on $(\End\vlb)^\zgen$ is an $\BN$-grading 
and $(\End\vlb)^{\zgen}_h(0)=\bbk{\cdot}\mathsf{Id}_{\vlb}$. 
\item[\sf (ii)] \  the associative algebra $(\End\vlb)^\zgen$ contains no non-trivial idempotents and, thereby,
$\zgen$ is a \ri\ subalgebra of $\g$.
\end{itemize}
\end{thm}
\begin{proof}
(i) We have $\End\vlb\simeq\vlb\otimes\vlb^*=\bigoplus_{i=1}^k\sfr(\lb_i)$, where all $\lb_i\in\eus Q\cap\fxpl$, and 
we may assume that $\lb_1=0$, while
$\lb_i\ne 0$ for $i\ge 2$.  Then 
\[
      (\End\vlb)^{\zgen}=\sfr(0)\oplus \bigl( \bigoplus_{i=2}^k \sfr(\lb_i)^{\zgen}\bigr) .
\]
It follows from Proposition~\ref{prop:even-normal} that the $h$-grading of $(\End\vlb)^{\zgen}$ is 
non-negative and the 
component of grade $0$ is just $\sfr(0)=\bbk{\cdot}\mathsf{Id}_{\vlb}$.

(ii) By Lemma~\ref{lm:ass-alg-grad}, the $h$-grading of  $(\End\vlb)^{\zgen}$ is compatible with
the structure of the associative algebra, and  by Lemma~\ref{lm:no-idempot},
$(\End\vlb)^{\zgen}$ contains no nontrivial idempotents. Thus, $\vlb$ is $\zgen$-indecomposable,
and thereby $\zgen$ is \ri.
\end{proof}

\begin{rmk}   \label{rem:cent-nilp-wide}
Using the classification of the nilpotent $G$-orbits in $\g$, one can verify that $\dim\zgen\ge \rk(\g)$ for any non-degenerate $e\in\fN$, and $\dim\zgen = \rk(\g)$ if and only if $e$ is a regular (=principal) nilpotent element.
Moreover, $\zgen$ is abelian if and only if $e$ is regular. It would be interesting to have a conceptual
proof for these observations.
\end{rmk}

\begin{rmk}   \label{rem:simple-semis2}
It can happen that $f$ and a {\bf proper} subalgebra $\ah\subset\zgen$ generate the whole of $\g$.
Then the above reasoning applies, and $\ah$ appears to be \ri. An instance of this phenomenon is
provided in Example~\ref{ex:wide-2-dim}.
\end{rmk}

\section{Some regular \textsf{ad}-nilpotent subalgebras are \ri}
\label{sect:pi-part}

\noindent
A subalgebra $\h\subset\g$ is said to be {\it regular}, if it is normalised by a Cartan subalgebra of 
$\g$. Without loss of generality, one may only consider regular subalgebras such that 
$[\te,\h]\subset\h$ for our fixed $\te$. We additionally assume below that $\h$ is 
\textsf{ad}-nilpotent. Then $\h=\bigoplus_{\gamma\in \Delta_\h} \g_\gamma$, where $\Delta_\h$ 
is a closed subset of $\Delta$ and $\Delta_\h \cap (-\Delta_\h)=\varnothing$. (Note that we do \un
{not} assume here that $\Delta_\h\subset\Delta^+$.)
Recall that a subset $\Gamma\subset\Delta$ is {\it closed\/} if whenever $\gamma_1,\gamma_2\in \Gamma$ and 
$\gamma_1{+}\gamma_2\in\Delta$, then $\gamma_1{+}\gamma_2\in\Gamma$;
the {\it closure\/} of  $\Gamma$ is the smallest closed subset of $\Delta$ containing $\Gamma$.
Write $\vlb_\mu$ for the $\mu$-weight space of  $\vlb$.
As is well-known 
\cite[Ch.\,VIII, \S\,7]{bour7-8},
\[
    \vlb_0\ne \{0\} \ \Leftrightarrow \ \lb \in \eus Q \ \Leftrightarrow \ \text{ all the weights of $\vlb$ belong to $\eus Q$}.
\]

\begin{lm}   \label{lm:conus1}  
Let $\h$ be as above. Then
\begin{itemize}
\item[\sf (i)] \ 
for any $\lb\in \fxpl$, we have 
$  \vlb^{\h}\subset   \bigoplus_{\mu\in \eus C} \vlb_\mu$, 
where $\eus C(\h)=\{\mu\in \eus E\mid (\mu, \gamma)\ge 0 \ \ \forall\gamma\in \Delta_\h \}$ is a  closed cone in $\eus E$, which does not depend on $\lb$;
\item[\sf (ii)] \ suppose that the closure of $\Delta_\h\cup (-\Delta_\h)$ equals $\Delta$.
Then $\eus C(\h)$ is a strictly convex cone and $\vlb^{\h}\subset   \bigoplus_{\mu\in \eus C(\h)\setminus\{0\}} \vlb_\mu$ for any nonzero
$\lb\in\fxpl$.
\end{itemize}
\end{lm}
\begin{proof}
(i) If $\gamma\in \Delta$, then $\vlb^{\g_\gamma}\subset 
\bigoplus_{\mu: \ (\mu,\gamma)\ge 0} \vlb_\mu$, see \cite[Ch.\,VIII, \S\,7]{bour7-8}. Therefore,
\[
   \vlb^{\h}\subset \bigcap_{\gamma\in \Delta_\h}\vlb^{\g_\gamma}
   \subset \bigoplus_{\mu\in \eus C(\h)} \vlb_\mu ,
\]
(ii) If the closure of $\Delta_\h\cup (-\Delta_\h)$ equals $\Delta$, then $\Delta_\h$ contains a basis for
$\eus E$ and hence $\eus C(\h)$ is strictly convex. Therefore, it remains to prove that even if 
$\vlb_0\ne 0$ (i.e., $\lb\in\eus Q$), then still $\vlb_0^{\h}=0$. Indeed, 
\[
  \vlb_0^{\h}= 
  \bigcap_{\gamma\in\Delta_\h}\Ker\bigl(\ad e_\gamma: \vlb_0 \to \vlb_\gamma\bigr).
\]
But it follows from the $\tri$-theory applied to the subalgebra generated by $\g_\gamma$ and 
$\g_{-\gamma}$ that 
\[
      \Ker\bigl( \ad e_\gamma: \vlb_0 \to \vlb_\gamma\bigr) = \Ker\bigl( \ad e_{-\gamma}: \vlb_0 \to \vlb_{-\gamma}\bigr) .
\]
Therefore,  $\vlb_0^{\h}$ is also a fixed point subspace of the subalgebra generated
by $\te$ and all $\g_\gamma, \g_{-\gamma}$ with $\gamma\in\Delta_\h$. The hypothesis on the closure
implies that this subalgebra equals $\g$. Hence
$\vlb_0^{\h}=\vlb^\g=\{0\}$ if $\lb\ne 0$.
\end{proof}

\begin{thm}  \label{thm:main2}
Let $\h$ be an \textsf{ad}-nilpotent subalgebra of $\g$ normalised by $\te$ and $\Delta_\h$ the corresponding set of roots.  
\begin{itemize}
\item[\sf (i)] \ 
Suppose that the closure of $\Delta_\h\cup (-\Delta_\h)$ equals $\Delta$. Then,
for any $\lb\in\fxpl$, the associative algebra  $(\End\vlb)^{\h}$ does not contain non-trivial 
idempotents; hence $\vlb$ is $\h$-indecomposable and thereby  $\h$ is \ri.
\item[\sf (ii)] \ 
Conversely, if\/ $\h$ is \ri, then the closure of $\Delta_\h\cup (-\Delta_\h)$ equals $\Delta$.
\end{itemize}
\end{thm}
\begin{proof}
(i) \ As before, $\End\vlb\simeq\vlb\otimes\vlb^*=\bigoplus_{i=1}^k\sfr(\lb_i)$, where all $\lb_i\in\eus Q\cap\fxpl$, and 
we may assume that $\lb_1=0$, while
$\lb_i\ne 0$ for $i\ge 2$.  Then 
\[
      (\End\vlb)^{\h}=\sfr(0)\oplus \bigl( \bigoplus_{i=2}^k \sfr(\lb_i)^{\h}\bigr) .
\]
By Lemma~\ref{lm:conus1}, each $\sfr(\lb_i)^{\h}$
is $\eus C(\h)$-graded and the component
of grade $0$ is just $\sfr(0)=\bbk{\cdot}\mathsf{Id}_{\vlb}$. 
Because this grading of $(\End\vlb)^{\h}$
is determined by weights of $\te$ and these weights are contained in the strictly convex
cone $\eus C(\h)$, it is an associative algebra grading and 
the only idempotent sitting in $(\End\vlb)^{\h(\Pi')}$ is $\mathsf{Id}_{\vlb}$ 
(see Lemma~\ref{lm:no-idemp-multi}). Thus, $\vlb$ is $\h$-indecomposable, and we are done.

(ii) \ Let $\tilde\Delta$ be the closure of $\Delta_\h\cup (-\Delta_\h)$. Assume that 
$\tilde\Delta\ne \Delta$. Then $\tilde\g:=\te\oplus(\bigoplus_{\gamma\in\tilde\Delta}\g_\gamma)$ is 
a proper reductive subalgebra of $\g$ and $\h\subset\tilde\g$. Hence the simple $\g$-module $\g$ is
decomposable as $\tilde\g$- and $\h$-module.
\end{proof}

In the rest of the section, we consider important examples illustrating Theorem~\ref{thm:main2}.

\begin{ex}[Parabolic subalgebras]  \label{ex:parabolic}
Let $\n$ be the nilradical of a standard parabolic subalgebra $\p$. It is easily seen that if
$\p$ contains no simple ideals of $\g$, then
the closure of $\Delta_\n\cup (-\Delta_\n)$ equals $\Delta$. Therefore, 
Theorem~\ref{thm:main1-pos-grad} follows from Theorem~\ref{thm:main2}(i). But we 
include a separate treatment for the nilpotent radicals, because it does not require multigradings
and yields a more complete information. 
\end{ex}

\begin{ex}[The derived algebra of $\ut^+$]  \label{ex:derived}
For the \textsf{ad}-nilpotent subalgebra $\tilde\ut:=[\ut^+,\ut^+]$, we have 
$\Delta_{\tilde\ut}=\Delta^+\setminus\Pi$. If $\g$ has no simple ideals $\tri$ or $\mathfrak{sl}_3$, then the closure of 
$(\Delta^+\setminus\Pi)\cup (-(\Delta^+\setminus\Pi))$ is $\Delta$. Hence $\tilde\ut$ is \ri\ in all these cases.

By \cite[Sect.\,4]{imrn10},  the cone $\eus C(\tilde\ut)$ is generated by the weights
$\vp_\ap, \vp_\ap-\ap \ (\ap\in\Pi)$; and it also follows from 
\cite[Sect.\,1]{imrn10} that 
$\vlb^{\tilde\ut}$ is positively $\rho^\vee$-graded, where
$\rho^\vee=\frac{1}{2}\sum_{\gamma\in\Delta^+} \gamma^\vee$.
\end{ex}

\begin{ex}[$\Pi$-partition subalgebras]  \label{ex:pi-partition}
Let $\Pi'$ be a subset of $\Pi$. As the following exposition is symmetric with respect to 
$\Pi'$ and $\Pi''=\Pi\setminus\Pi'$, it is convenient to think of it as a partition
$\Pi=\Pi'\sqcup \Pi''$.

A $\Pi$-{\it partition subalgebra\/} of $\g$ is the Lie algebra generated by the root spaces 
$\g_\ap$ ($\ap\in\Pi'$) and $\g_{-\ap}$ ($\ap\in\Pi''$). Write  $\h(\Pi')$ for this subalgebra. 

Here are some simple observations related to these subalgebras:

\textbullet\quad $\dim\h(\Pi')\ge \rk(\g)$ and $\te$ normalises $\h(\Pi')$;

\textbullet\quad $\h(\Pi')\simeq \h(\Pi'')$ (use the Weyl involution of $\g$);

\textbullet\quad $\h(\Pi)=\ut^+$ and $\h(\varnothing)=\ut^-$;

\textbullet\quad The weights $\Pi'\cup (-\Pi'')$ are contained in an open half-space of $\eus E$.

\noindent
The last property implies that $\h(\Pi')$ is contained in a maximal nilpotent  subalgebra of $\g$. 
Hence $\h(\Pi')$ consists of nilpotent  elements and $\Delta_{\h(\Pi')}\cap (-\Delta_{\h(\Pi')})=
\varnothing$. 

Since $\Delta_{\h(\Pi')}\cup (-\Delta_{\h(\Pi')})\supset \Pi\cup (-\Pi)$, the closure of
$\Delta_{\h(\Pi')}\cup (-\Delta_{\h(\Pi')})$ is $\Delta$. Hence Theorem~\ref{thm:main2}(i)
applies here and all $\Pi$-partition subalgebras are \ri.

The most interesting $\Pi$-partition subalgebra occurs if 
the roots in $\Pi'$ are pairwise orthogonal (=\,disjoint on the Dynkin diagram) and the same property also
holds for $\Pi''$. Since the Dynkin diagram is a tree, 
such a partition of $\Pi$ 
is unique, so there are two (isomorphic) respective subalgebras of
$\g$. This partition of $\Pi$ is said to be {\it disjoint} and its parts are denoted by 
$\{\tilde\Pi,\tilde{\tilde\Pi}\}$. 
This discussion yields the following simple but useful assertion.

\begin{prop}
For a partition $\Pi=\Pi'\sqcup \Pi''$, the following conditions are equivalent:

1) \ this partition is disjoint, i.e., $\Pi'=\tilde\Pi$ or $\tilde{\tilde\Pi}$;

2) \ $\dim\h(\Pi')=\rk(\g)$;

3) \ $\h(\Pi')$ is abelian;

4) \ $\Delta_{\h(\Pi')}=\Pi'\cup (-\Pi'')$.
\end{prop}

In \cite{casati}, it is proved that $\h(\tilde\Pi)$ is \ri\ for $\g=\slno$. But that proof is rather
technical and exploits Littelmann's theory of standard bases for the $\slno$-representations.
Our approach  provides a much simpler proof for a much stronger result (Theorem~\ref{thm:main2}).
\end{ex}

\section{On a general approach to \ri\ subalgebras and indecomposable representations}
\label{sect:general}

\subsection{Simple properties} 
Here we discuss some general properties of \ri\ subalgebras of $\g$ and related problems.

\begin{lm}    \label{lm:simple-prop}
\leavevmode \par
\begin{itemize}  
\item[\sf (i)] \ If\/ $\ah_1\subset\ah_2$ and $\ah_1$ is \ri, then so is $\ah_2$;
\item[\sf (ii)] \ If\/ $\ah\subset \es\subsetneqq \g$ and $\es$ is reductive, then $\ah$ is not \ri;
\item[\sf (iii)] \ If\/ $\ah$ is \ri, then $\z_\g(\ah)$ is an \textsf{ad}-nilpotent subalgebra.
\end{itemize}
\end{lm}
\begin{proof}
(i) \ Obvious.

(ii) \  The simple $\g$-module $\g$ is decomposable as $\es$-module.

(iii) \ 
If $s\in\z_\g(\ah)$ is semisimple, then $\ah$ is contained in the reductive subalgebra
$\z_\g(s)$, hence $\ah$ is not \ri. That is, $\z_\g(\ah)$ does not contain semisimple elements.
As $\z_\g(\ah)$ is an algebraic Lie algebra \cite[7.4]{bor69}, it contains the semisimple part of every
element. Therefore, $\z_\g(\ah)$ must contain only nilpotent elements.
\end{proof}

All \ri\ subalgebras occurring in Section~\ref{sect:pi-part} are of dimension at least $\rk(\g)$
(see also Remark~\ref{rmk:min-nil-rad}), and the same is true for the nilpotent radicals of centralisers of non-degenerate nilpotent elements, see Section~\ref{sect:reg-nilp}.
Moreover, in both cases, the subalgebras of dimension $\rk(\g)$ are necessarily abelian.
A partial explanation is given by 

\begin{lm}
Suppose that $\ah$ is \ri\ and regular. Then $\dim\ah \ge\rk(\g)$ and if\/ $\dim\ah =\rk(\g)$, then
$\ah$ is \textsf{ad}-nilpotent and abelian.
\end{lm}
\begin{proof}
If $[\te,\ah]\subset \ah$, then $\ah=\tilde\te\oplus(\bigoplus_{\gamma\in\Delta_\ah}\g_\gamma)$, where $\tilde\te$ is a subspace of $\te$ such that  $\tilde\te\supset  \Delta_\ah\cap(-\Delta_\ah)$
(in the last embedding we identify $\te$ and $\te^*$).
If $\dim\ah< \rk(\g)$, then $\#\Delta_\ah < \rk(\g)$ and
$\z_\g(\ah)$ certainly contains a nonzero element of $\te$, i.e., $\ah$ cannot be \ri.
If $\dim\ah =\rk(\g)$, then a similar argument shows that we must have  $\#\Delta_\ah = \rk(\g)$,
the elements of $\Delta_\ah$ are linearly independent and $\tilde\te=0$. Moreover, since
$\Delta_\ah$ is linearly independent and closed, $\ah$ is abelian.
\end{proof}
However,  $\rk(\g)$ provides the strict lower bound only  for the dimension of {\bf regular} wide 
subalgebras of $\g$. We prove below that every simple Lie algebra has a wide commutative
subalgebra of dimension~2.
Note also that it may happen that $\ah$ is not \ri, but there still exist families of 
$\ah$-inde\-com\-posable simple $\g$-modules. Here is a sample reason for such phenomenon.

\begin{lm}   \label{prop:a-in-s-in-g}
Let $\tilde\g\subset\g$ be a proper semisimple subalgebra and $\ah$ is wide in $\tilde\g$.
Suppose that a simple $\g$-module $\vlb$ remains simple as $\tilde\g$-module. Then
$\vlb$ is $\ah$-indecomposable.
\end{lm}

\subsection{Wide subalgebras and epimorphic subgroups} 
A subgroup $H\subset G$ is said to be {\it epimorphic}, if $\bbk[G]^H=\bbk$. Equivalently, $H$ is 
epimorphic if $\vlb^H=\{0\}$ unless $\lb=0$ \cite{bb92}. One easily proves that $H$ is epimorphic if 
and only if the identity component of $H$ is. Therefore, we may say that a subalgebra $\h$ is  
{\it epimorphic\/} if $\h=\Lie(H)$ and $H$ is epimorphic in the above sense.

By \cite[Theorem\,1]{bb92}, $\h$ is epimorphic if and only
if the following condition holds: {\sl If\/ $\VV$ is a $\g$-module and\/ $\VV=V_1\oplus V_2$ is a sum 
of\/ $\h$-modules, then $V_1$ and $V_2$ are actually $\g$-invariant.}
Compare this with the definition of a wide subalgebra, which requires indecomposability only for the simple $\g$-modules!

This implies that any epimorphic subalgebra is wide. Alternatively, one may notice that if 
$\h$ is epimorphic, then $(\End\vlb)^\h=\bbk{\cdot}\mathsf{Id}_{\vlb}$ for all $\lb\in\fxpl$ 
and hence $\h$ is \ri. 
There is a close relationship between regular wide and epimorphic subalgebras.

\begin{prop}
Suppose that the subalgebra $\h\subset\g$ is \textsf{ad}-nilpotent and $[\te,\h]\subset\h$. Then\/ $\h\oplus\te$ is epimorphic if and only if\/ $\h$ is wide.
\end{prop}
\begin{proof}
Here $\h$ is the nilpotent radical of the regular solvable subalgebra $\tilde\h=\h\oplus\te$. 
\\
By \cite[Korollar\,3.6]{pomm}, if $\tilde\h$ is epimorphic, then the closure
of $\Delta_\h\cup(-\Delta_\h)$ is $\Delta$. Hence $\h$ is wide in view of Theorem~\ref{thm:main2}(i).
\\
Conversely, if $\h$ is wide, then the closure of $\Delta_\h\cup(-\Delta_\h)$ is $\Delta$ according 
to Theorem~\ref{thm:main2}(ii), and again \cite[Korollar\,3.6]{pomm} shows that $\h\oplus\te$ 
is epimorphic.
\end{proof}

Any simple Lie algebra contains a three-dimensional solvable epimorphic subalgebra 
(see~\cite[n.\,5(b)]{bb92}), but this subalgebra is neither regular nor \textsf{ad}-nilpotent.
Below, we recall the construction and show that the two-dimensional nilradical of that subalgebra is \ri.

\begin{ex}     \label{ex:wide-2-dim}
Let $\es:=\langle e,h,f\rangle$ be a principal $\tri$-subalgebra of a simple Lie algebra $\g$. Then 
$\es$ is not contained in a proper regular semisimple subalgebra of $\g$ \cite[Theorem\,9.1]{dy52a}. 
Actually, $\es$ is either a maximal semisimple subalgebra, or is contained in a unique maximal proper 
semisimple subalgebra $\tilde\g$ of $\g$, see  \cite{dy52b} for the classical Lie algebras and 
\cite{dy52a} for the exceptional algebras. For instance, if $\g$ is of type $\GR{E}{6}$, then
$\tilde\g$ is of type $\GR{F}{4}$, whereas for all other exceptional algebras, one has $\es=\tilde\g$ 
\cite[Theorem\,15.2]{dy52a}. By Lemma~\ref{lm:generated}, $\tilde\g$ cannot contain the whole of 
$\g^e=\zgen$.
Therefore, one can pick an $h$-eigenvector  $\tilde e\in\g^e$ 
such that $\es$ and $\tilde e$ generate the whole of $\g$. In other words, $f$ and the commutative
subalgebra $\ah=\langle e,\tilde e\rangle$ generate the whole of $\g$. Applying then 
Proposition~\ref{prop:even-normal} and Theorem~\ref{thm:zge-pos-grad} to $\ah$ 
(in place of $\zgen$), we conclude that $\ah$ is wide. Here $\langle h,e,\tilde e\rangle$ is an 
epimorphic subalgebra of $\g$ described in \cite{bb92}, and $\ah$ is its nilradical.
\end{ex}
This prompts the following 
\begin{quest}  \label{vopr2}
Let $\ah$ be an epimorphic algebraic subalgebra of $\g$. Is it true that the nilpotent radical 
$\ah_{nil}$ is \ri?  \end{quest}

\subsection{Example: the euclidean Lie algebra $\ee_3$} 
Following \cite{duglas}, we denote by $\ee_3$ the semi-direct product 
$\mathfrak{so}_3\ltimes \bbk^3\simeq \tri\ltimes\tri$. It is proved in~\cite{duglas} that, for
a certain embedding $\ee_3\subset \mathfrak{sl}_4$, the simple $\mathfrak{sl}_4$-modules
$\sfr(m\vp_1)$ and $\sfr(m\vp_3)$ are $\ee_3$-indecomposable for all $m\in\BN$, whereas 
$\sfr(\vp_2)$ and $\sfr(2\vp_2)$ are not. [We use the obvious numbering of the fundamental 
weights of $\mathfrak{sl}_4$.] To illustrate the usefulness of our methods, 
we provide a simpler derivation (and a generalisation) of those results.

The embedding $\ee_3\subset \mathfrak{sl}_4$ is given by Equations (3.1) and (4.1) in~\cite{duglas}.
Making a suitable permutation of the corresponding basis vectors of $\bbk^4$, one easily finds that 
$\ee_3$
can be regarded as the subalgebra 
\beq   \label{eq:imbed-e3}
     \left\{ \begin{pmatrix} A & B\\ 0 & A \end{pmatrix} \mid A,B\in \tri \right\} \subset \mathfrak{sl}_4 .
\eeq
Let $\psi$ be the skew-symmetric bilinear form on $\bbk^4$ with the matrix
\[
 \Psi=\begin{pmatrix} 0 & 0 & 0 & 1 \\ 0 & 0 & -1 & 0 \\ 0 & 1 & 0 & 0\\ -1 & 0 & 0 & 0
  \end{pmatrix}
\]
and let $\mathfrak{sp}_4$ denote the stabiliser of the form $\psi$. That is,
$
      \mathfrak{sp}_4=\{ g\in \mathfrak{sl}_4 \mid \Psi g+g^t\Psi=0\} 
$.
Then one readily verifies that $\ee_3\subset \mathfrak{sp}_4$. Hence $\ee_3$ is not \ri\ in $\mathfrak{sl}_4$, in view 
of Lemma~\ref{lm:simple-prop}(ii). However, 

\begin{lm}        \label{lm:wide-in-sp}
 $\ee_3$ is  \ri\ in $\mathfrak{sp}_4$.
\end{lm}
\begin{proof}
Let $\tilde\ap_1,\tilde\ap_2$ be the simple roots of $\mathfrak{sp}_4$ ($\tilde\ap_1$ is short) and $\p(2)$ the parabolic subalgebra of $\mathfrak{sp}_4$ corresponding to $\Pi'=\{\tilde\ap_2\}$ (see notation of
Section~\ref{sect:radical}). Then $\ee_3$ is of codimension 1 in $\p(2)$ and 
$\ee_3\supset \p(2)_{nil}$.  In the above Eq.~\eqref{eq:imbed-e3}, 
$\p(2)_{nil}$ is the set of matrices with ${A=0}$.
By Theorem~\ref{thm:main1-pos-grad} and Lemma~\ref{lm:simple-prop}(i), we conclude that 
$\ee_3$ is  \ri\ in $\mathfrak{sp}_4$.
\end{proof}

\begin{thm}
The simple $\mathfrak{sl}_4$-module $\vlb$ is $\ee_3$-indecomposable if and only if 
$\lb\in\{m\vp_1, m\vp_3\}$ with any $m\in\BN$.
\end{thm}
\begin{proof}
As is well-known, the simple $ \mathfrak{sl}_4$-modules  $\sfr(m\vp_1)$ and $\sfr(m\vp_3)$
remain simple as $\mathfrak{sp}_4$-modules. Hence they are $\ee_3$-indecomposable.
On the other hand, all other simple $ \mathfrak{sl}_4$-modules are decomposable as 
$\mathfrak{sp}_4$-modules. This can be verified using Weyl's dimension 
formula~\cite[Ch.\,VIII, \S\,9.2]{bour7-8}.
Namely, if $\lb=a_1\vp_1+a_2\vp_2+a_3\vp_3$, then $\vlb\vert_{\mathfrak{sp}_4}$ contains
the simple $\mathfrak{sp}_4$-module with highest weight $\tilde\lb=(a_1+a_3)\tilde\vp_1+a_2\tilde\vp_2$.  Then Weyl's formula shows that 
$\dim\vlb >  \sfr(\tilde\lb)$ if $\lb\ne m\vp_1, m\vp_3$.

Alternatively, one can refer to the seminal work of E.B.~Dynkin on maximal subgroups.
Specifically, in \cite[Theorem\,4.1]{dy52b},
Dynkin describes all irreducible representations of $\sln$ that remain irreducible upon the restriction to
a semisimple subalgebra.
\end{proof}

\begin{rema}
The subalgebra $\mathfrak{sp}_4$ is symmetric in $\mathfrak{sl}_4$, and it is known that
$\vlb^{\mathfrak{sp}_4}\ne 0$ if and only if $\lb=m\vp_2$.
This again shows that $\sfr(m\vp_2)$ is decomposable as $\mathfrak{sp}_4$-module for all 
$m\in\BN$. 
\end{rema}

\end{document}